\documentclass[12pt]{article}
\usepackage{amsmath,amssymb,amsfonts,amsthm}
\newenvironment{myabstract}{\par\noindent
{\bf Abstract . } \small }
{\par\vskip8pt minus3pt\rm}
\newcounter{item}[section]
\newcounter{kirshr}
\newcounter{kirsha}
\newcounter{kirshb}
\newenvironment{enumroman}{\setcounter{kirshr}{1}
\begin{list}{(\roman{kirshr})}{\usecounter{kirshr}} }{\end{list}}
\newenvironment{enumarab}{\setcounter{kirshb}{1}
\begin{list}{(\arabic{kirshb})}{\usecounter{kirshb}} }{\end{list}}
\newtheorem{theorem}{Theorem}[section]

\newtheorem{corollary}[theorem]{Corollary}
\newenvironment{demo}[1]{\noindent{\bf #1.}\upshape\mdseries}
{\nopagebreak{\hfill\rule{2mm}{2mm}\nopagebreak}\par\normalfont}
\theoremstyle{definition}

\newtheorem{example}[theorem]{Example}
\newtheorem{definition}[theorem]{Definition}

\def\C{{\mathfrak{C}}}
\def\Fm{{\mathfrak{Fm}}}

\def\Nr{{\mathfrak{Nr}}}

\def\Fm{{\mathfrak{Fm}}}
\def\A{{\mathfrak{A}}}
\def\B{{\mathfrak{B}}}
\def\C{{\mathfrak{C}}}
\def\D{{\mathfrak{D}}}

\def\Sn{{\mathfrak{Sn}}}
\def\CA{{\bf CA}}

\def\K{{\bf K}}
\def\K{{\bf K}}

\def\(R)RA{{\bf (R)RA}}

\def\Zd{{\mathfrak Zd}}
\def\Co{{\sf Co}}
\def\Ig{{\sf Ig}}

 \def\CA{{\sf CA}}
\def\B{{\sf B}}
\def\G{{\sf G}}

\def\K{{\sf K}}

\def\Nr{{\mathfrak{Nr}}}

\def\Nr{{\mathfrak{Nr}}}

\def\A{{\mathfrak{A}}}
\def\B{{\mathfrak{B}}}
\def\C{{\mathfrak{C}}}
\def\D{{\mathfrak{D}}}

\def\A{{\mathfrak{A}}}
\def\B{{\mathfrak{B}}}
\def\C{{\mathfrak{C}}}
\def\D{{\mathfrak{D}}}

\def\L{{\mathfrak{L}}}

\def\L{{\mathfrak{L}}}
\def\CA{{\bf CA}}

\def\G{{\bf G}}

\title{Representations of bounded distributive lattices as continuous sections of sheaves, with an application to algebraic logic}
\author{Tarek Sayed Ahmed}

\begin{document}
\maketitle

\begin{myabstract} Using Sheaf duality theory of Comer for cylindric algebras, we give a representation theorem of of distributive bounded lattices expanded by modalities (functions distributing over joins)
as the continuous sections of sheaves. 
Our representation is defined via a contravariant functor.
We give applications to many-valued logics logics and  various modifications of  first order logic and multi-modal logic, set in an algebraic framework.
%study properties of free algebras of $BAO$s, for example when are they atomi
%we  obtain as special cases deep theorems for many 
%cylindric-like algebras.
\end{myabstract}

\begin{definition} A triple $(X, \leq \tau)$, where

\begin{enumarab}
\item $(X,\leq )$ is a partially ordered set
\item $(X, \tau)$ is a topological space, is called a Priestly space if

(a) $\tau$ is a Stone space,

(b) For any $x,y\in X$ such that $x\nleq y$ there is a downward clopen set $U$ such that $y\in U$ and $x\notin U$.
(Downward here, means that when $u\in U$ and $v\leq u$, then $v\in U$.
\end{enumarab}
\end{definition}
\begin{definition}
\begin{enumarab} 
\item A non-empty subset  $I$ of a partially ordered set $(P, \leq )$ is an ideal if the following conditions hold:

(a) For every $x\in I$, $y \leq x$ implies that $y\in I$ ($I$ is a lower set).

(b) For every $x,y\in L$ there is a $z\in I$ such that  $x\leq z$ and $y\leq z$ ($I$ is a directed set).
\item $I$ as above is a prime ideal if for every elements $x$ and $y$ in $P$, $x\land y\in P$ implies $x\in P$ or $y\in I.$ Here $x\land y$ denotes 
$inf\{x,y\}$; it is maximal if it is not properly contained in any proper ideal.

\item A lattice is simple if has only the universal congruence and the identity one.

\end{enumarab}
\end{definition}
Let $V$ be the class of bounded distributive lattices, and let $L\in V$.
We consider lattices as algebraic structurses $(L, \land , \lor, 0, 1)$.
Let $Spec(L)$ be the set of prime ideals endowed with the Priestly topology which has
a base of the form $N_a=\{P\in Spec(L): a\notin P\}$ and their complements.
This space is called the Priestly space corresponding to $L$, or simply, the Priestly space of $L$.
Let $Pries$ be the category of Priestly spaces, where morphisms are homeomorphisms.
We regard $V$ as a concrete category whose morphisms are algebraic homomorphisms (preserving the operations).

Let $F: V\to Pries$ be the functor that takes $L$ to its Priestly space, with the image of morphisms defined by $F(h(P))=h^{-1}(P)$. 
This is an adjoint situation, 
with its inverse the contravariant functor which assigns to 
to a Priestly space the lattice of clopen downward sets and images of morphisms are given via $\Delta(f)(U)= f^{-1}(U)$.

In algebraic logic quantifier logics, like first order logic and other variants therefore, quantifiers are treated as connectives. 
This has a modal formalism, as well, 
which views quantifies (and their duals) as boxes and diamonds, that is, as a multi -dimensional modal logic.
The algebraic framework of such muti-dimensional modal logics, or briefly multi-modal logics
when their propositional part is classical, is the notion of a Boolean algebra with operators $(BAOs)$.

In this note, we deal with bounded distributive lattices with operators (reflecting quantifiers), denoted by $BLO$s.
This notion covers a plathora of logics starting from many valued logic, 
fuzzy logic, intuitionistic logic, multi-modal logic, different versions (like extensions and reducts) of first order logic.

\begin{definition}

A $BLO$ is an algebra of the form $(L, f_i)_{i\in I}$ where $L$ is a distributive bounded lattice, $I$ is a set (could be infinite) 
and the $f_i$'s are unary operators that preserve order, and joins, and are idempotent $f_if_i(x)=f_i(x)$,  
on $L$, such that $f_i(0)=0$, $f_1(1)=1,$ 
and if $x\in L$, and $\Delta x=\{i\in I: f_i(x)\neq x\}$, then $\Delta( x\lor y)\subseteq \Delta x\cup \Delta y$ 
and same for meets. 
\end{definition}

\begin{definition} Let $\A=(L, f_i)_{i\in I}$ be a $BLO$. Then a subset $I$ of $\A$ is an ideal of $\A$, if $I$ is an ideal of $L$ 
and for all $i\in I$, and all $x\in L$, if $x\in I$, then $f_i(x)\in L$
\end{definition}
What distinguishes the algebraic treatment of logics corresponding to such $BLO$s, is their propositional part; 
it can be a $BL$ algebra, an $MV$ algebra, a Heyting algebra, a Boolean algebra 
and so forth.

Here we represent such structures 
as the continuous sections of sheaves; the representation in this geometric context is implemented by a contravariant functor.

We start by a concrete example addressing variants and extension first order logics. 
The following discussion applies to $L_n$ (first order logic with $n$ variables), $L_{\omega,\omega}$ (usual first order logic), 
rich logics, Keislers logics with and without equality, finitray logics of infinitary
relations; the latter three logics are infinitary extensions of first order logic, 
though the former and the latter have a finitary flavour, because quantification is taken only
on finitely many variables. These logics have an extensive literature in algebraic logic.
%It also applies to non classical logics, whose Stone space is the Zarski topology.
Let us start with the concrete example of usual first order logic. $\L_n$ denotes a relational first order language (we have no function symbols)
with $n$ constants, $n\leq \omega,$
and as usual a sequence of variables of order type $\omega$.

\begin{example}

Let $\Sn_{\L_n}$ denote the set of all $\L_n$ sentences, and  fix an enumeration $(c_i: i<n)$ of the constant symbols.
We assume that $T\subseteq  \Sn_{\L_0}$. 
Let $X_T=\{\Delta\subseteq \Sn_{\L_0}: \Delta \text{ is complete }\}$.
This is simply the underlying set of the Priestly  space, equivalently  the Stone space, 
of the Boolean algebra $\Sn_{L_0}/T$. For each $\Delta\in X_T,$ let $\Sn_{\L_n}/{\Delta}$
be the corresponding Tarski-Lindenbaum quotient algebra, which is a (representable) cylindric algebra of dimension $n$. 
The $i$th cylindrifier $c_i$ is defined by 
$c_i\phi/{\Delta}=\exists \phi(c_i|x)$, where the latter is the formula
obtained by replacing the $i$th constant if present by the first variable $x$ not occurring in $\phi$, 
and then applying the existential quantifier $\exists x$.
Let $\delta T$ be the following disjoint union
$\bigcup_{\Delta\in X_T}\{\Delta\}\times Sn_{\L_n}/{\Delta}.$ 
Define the following topologies, on $X_T$ and $\delta T$, respectively.
On $X_{T}$ the Priestly (Stone) topology, and on $\delta_{\Gamma}$ 
the topology with base $B_{\psi,\phi}=\{\Delta, [\phi]_{\Delta}, \psi\in \Delta, \Delta\in \Delta_{\Gamma}\}.$
Then $(X_T, \delta T)$ is a {\it sheaf}, and its dual consisting of the continuous sections, 
$\Gamma(T,\Delta)$, with operations defined pointwise, is actually isomorphic to $\Sn_{\L_n}/T$. 

\end{example}

\begin{example}
By the same token, let $\L$ be the  predicate language for $BL$ algebras, $\Fm$ denotes the set of $L$ formulas, and $\Sn$ denotes the 
set of all sentences (formulas with no free variables). 
This for example includes $MV$ algebras; that are, in turn, algebraisations of many 
valued logics.
Let $X_T$ be the Zarski (equivalently the Priestly) topology on $\Sn/T$ based on $\{\Delta\in Spec(\Sn): a\notin \Delta\}$. 
Let $\delta T=\bigcup _{\Delta\in X_T}\{\Delta\}\times \Fm_{\Delta}$. 
Then again, we have $(X_T, \delta T)$ is a {\it sheaf}, and its dual consisting of the continuous sections with operations defined pointwise, 
$\Gamma(T,\Delta)$ is actually isomorphic to $\Fm_T$. 
\end{example}

This situation is very similar to the one in algebraic geometry of desribing the ring associated with the affine variety 
in terms of the local rings given at at point of the variety.

This needs further clarification. Let us formalize the above concrete examples in an abstract more general setting, that allows further applications.
Let $\A$ be a bounded distributive lattice with extra operations $(f_i: i\in I)$. $\Zd\A$ denotes the distributive 
bounded lattice $\Zd\A=\{x\in \A: f_ix=x,\  \forall i\in I\}$, where the operations are the natural restrictions.(Idempotency of the $f_i$s guarantees 
that this is well defined). 
If $\A$ is a locally finite  algebra of formulas of first order logic 
or predicate modal logic or intiutionistic logic, or any predicate logic where the $f_i$s are interpreted as the existential 
quantifiers, then $\Zd\A$ is the Boolean 
algebra of sentences.

Let $\K$ be class of bounded distributive lattices with extra operations $(f_i: i\in I)$.
We describe a functor that associates to each $\A\in \K$, and $J\subseteq I$, a pair of topological spaces
$(X(\A,J), \delta(\A))=\A^d$, where $\delta(\A)$ has an algebraic structure, as well; in fact it is a subdirect product
of distributive lattices, that turn out to be simple (have no proper congruences)
under favourable circumstances, in which case $\delta(\A)$ is a semi-simple lattice carrying 
a product topology. This pair is  called the dual space of $\A$.
For $J\subseteq I$, let $\Nr_J\A=\{x\in A: f_ix=x \forall i\notin j\}$, with operations $f_i: i\in J$.
$X(\A,J)$ is the usual dual space of $\Nr_J\A$, that is, the set of all prime ideals of the lattice $\Nr_J\A$, 
this becomes a Priestly space (compact, Hausdorff and totally disconnected), when we take the collection of all sets
$N_a=\{x\in X(\A,J): a\notin x\}$, and their complements, as a base for the topology. 

For a set $X$ of an algebra $\A$ we let $\Co^{\A} X$ denote the congruence relation generated by $X$ (in the universal algebraic sense).
This is defined as the intersection of all congruence relations that have $X$ as an equivalence class.
%It is worthy of note, at this point that ideals and congruences 
Now we turn to defining the second component; this is more involved. 
For $x\in X(\A,J)$, let $\G_x=\A/\Co^{\A}x$  and 
$\delta(\A)=\bigcup\{\G_x: x\in X(\A)\}.$
This is clearly a disjoint union, and hence 
it can also be looked upon as the following product $\prod_{x\in \A} \G_x$ of algebras. 
This is not semi-simple, because $x$ is only prime,
least maximal in $\Nr_J\A$. 
But the semi-simple case will deserve special attention.

The projection $\pi:\delta(\A)\to X(\A)$ is defined for $s\in \G_x$ by $\pi(s)=x$.Here $\G_x=\pi^{-1}x$ is the stalk over $x$. For $a\in A$, 
we define a function
$\sigma_a: X(\A)\to \delta(\A)$ by $\sigma_a(x)=a/\Ig^{\A}x\in \G_x$. 

Now we define the topology on 
$\delta(\A)$. It is the smallest topology  for which all these functions are open, so $\delta(\A)$ 
has both an algebraic structure and a topological one, and they are compatible.
%Indeed, $\A^d=(X(\A), \delta(\A))$ is a reduced space.  

We can turn the glass around. Having such a space we associate a bounded distributive lattice in $\K$.
Let $\pi:\G\to X$ denote the projection associated with the space $(X,\G)$, built on $\A$.
A function  $\sigma:X\to \G$ is a section of $(X,\G)$ if $\pi\circ \sigma$ is the identity 
on $X$. 

Dually, the inverse construction  uses the sectional functor.
The set $\Gamma(X,\G)$ of all continuous sections of 
$(X,\G)$ becomes a $BLO$ by defining the operations pointwise, recall that $\G=\prod \G_x$ is a product of bounded distributive lattices.

The mapping $\eta:\A\to \Gamma(X(\A,J), \delta(\A))$ defined by $\eta(a)=\sigma_a$ 
is as easily checked  an isomorphism. 
Note that under this map an element in $\Nr_J\A$ corresponds with the characteristic 
function $\sigma_N\in \Gamma(X, \delta)$ 
of the basic set $N_a$.

To complete the definition of the contravariant 
functor we need to define the dual of morphisms. 

Given two spaces $(Y,\G)$ and $(X,\L)$ a sheaf morphism $H:(Y,\G)\to (X,\L)$ is a pair $(\lambda,\mu)$ where $\lambda:Y\to X$ is a continous map
and $\mu$ is a continous map $Y+_{\lambda} \L\to \G$ such that $\mu_y=\mu(y,-)$ is a homomorphism of $\L_{\lambda(y)}$ into $\G_y$.
We consider $Y+_{\lambda} \L=\{(y,t)\in Y\times \L:\lambda(y)=\pi(t)\}$ as a subspace of $Y\times \L$.
That is, it inherits its topology from the product topology on $Y\times \L$.

A sheaf morphism $(\lambda,\mu)=H:(Y,\G)\to (X,\L)$ produces a homomorphism of lattices
$\Gamma(H):\Gamma(X,\L)\to \Gamma(Y,\G)$ the natural way:
for $\sigma\in \Gamma(X,\L)$ define $\Gamma(H)\sigma$ by $(\Gamma(H)\sigma)(y)=\mu(y, \sigma(\lambda y))$ for all $y\in Y$.
A sheaf morphism $h^d:\B^d\to \A^d$ can also be asociated with a homomorphism $h:\A\to \B$. 
Define $h^d=(h^*, h^o)$ where for $y\in X(\B)$, $h^*(y)=h^{-1}\cap Zd\A$ and for $y\in X(\B)$ and $a\in A$
$$h^0(h, a/\Ig^{\A}h^*(y))=h(a)/\Ig^{\B}y.$$

%Because prime and maximal ideals are different we have:

\begin{example} Let $\A=\prod_{i\in I}\B_i$, where $\B_i$ are directly indecomposable $BAO$s. Then
$\Zd\A= {}^I2$ and $X(\A)$ is the Stone space of this algebra.
The stalk $\delta_{M}(\A)$ of $\A^{\delta}$ over $M\in X(\A)$ is the ultraproduct 
$\prod_{i\in I}\B_i/F$  where $F$ is the ultrafilter on $\wp(I)$  corresponding to $M$.
\end{example}

\begin{definition} Let $\A\in \CA_{\omega}$ and $x\in A$. The dimension set of $x$, in symbols $\Delta x$, 
is the set $\{i\in \omega: c_ix\neq x\}.$ Let $n\in \omega$. 
Then the $n$ neat reduct of $\A$ is the cylindric algebra of dimension $n$ 
consisting only of $n$ dimensional elements (those elements such that $\Delta x\subseteq n)$,
 and with operations indexed up to $n$.
\end{definition}
\begin{example}

\begin{enumarab}

\item Let $\A\in \Nr_n\CA_{\omega}$. Then there is a sheaf ${\bf X} =(X, \delta, \pi)$ such that $\A$ is isomorphic to continous sections
$\Gamma(X;\delta)$ of $\bold X.$  Indeed, let $X(\A)$ be the Stone space of $\Zd\A$. Then for any maximal ideal 
$x$ in $\Zd\A$, $\Ig^{\A}(x)$ is maximal in $\Nr_n\A$.
Let $\delta(A)=\bigcup \G_x$, where $\G_x=\A/\Ig^{\A}x$. 
The projection $\pi:\delta(\A)\to X(\A)$ is defined 
for $s\in \G_x$ by $\pi(s)=x$. For $a\in A$, 
we define a function $\sigma_a: X(\A)\to \delta(\A)$ by $\sigma_a(x)=a/\Ig^{\A}x\in \G_x$. 
Then $\pi\circ  \sigma$ is the identity and $\delta(\A)$ has  the smallest topology such that these maps are continuous.
Then $\eta: \A\to \Gamma(X(\A)), \delta)$ defined by
$\eta(a)=\sigma_a$ is the desired isomorphism.

\item Let $\A\in \Nr_n\CA_{\omega}$. For any ultrafilters 
$\mu$ and $\Gamma$ in $\Zd\A$, the map $\lambda:\A/\mu\to \A/\Gamma$ defined via, 
$a/\mu\mapsto a/\Gamma$
maps $\Zd\A$ into $\Zd\A$. (The latter is the set of zero-dimensional elements). 
The dual morphism is $\lambda^d=(\lambda, \lambda^0) :(X_{\Gamma}, \delta(\Gamma))\to (X_{\mu}, \delta(\mu))$, 
is defined by $\lambda(\Delta)=\Delta$ and $\lambda^0(\Delta, (\Delta), a/{\Delta}))=(\Delta, a/\Delta)$.
Thus it is an isompphism from $(X_{\Gamma}, \delta(\Gamma)$ onto the restriction of 
$(X_{\mu}, \delta(\mu))$ to the closed set $ X_{\Gamma}$.
Conversley, every restriction of $(X_{\mu}, \delta(\mu))$ to a closed subset $Y$ of $X_{\mu}$ is up to isomorphism the dual space of 
$\Nr_n\A/F$ for a filter $F$ of $\Zd\A$. For if $\Gamma=\bigcap Y$, then $Y=X_{\Gamma}$ 
since $Y$ is closed and the dual space of $\Nr_n\A/{\Gamma}$ is isomorphic to 
$(Y, \delta(\mu)\upharpoonright Y)$.
\end{enumarab}
\end{example}

For an algebra $\A$ and $X\subseteq \A$, $\Ig^{\A}X$ is the ideal generated by $X$. We write briefly lattice for a $BLO$; 
hopefully no confusion is likely to ensue.
\begin{definition}
\begin{enumarab}
\item  A lattice $L$ is regular if whenever $x$ is a prime ideal in $\Zd L$, then $\Ig^{\A}x$ is a prime ideal in
$\A$.

\item A lattice $L$ is strongly regular, if whenever $x$ is a prime idea in $\Zd \L$, then $\Ig^{\A}x$ is a maximal ideal in $\A$.

\item A lattice $L$ is congruence strongly regular, if whenever $x$ is a prime ideal in $\Zd\L$, then $\Co^{\A}x$ is a maximal congruence of $\A$.
%\item A lattice $L$ is 
\end{enumarab}
\end{definition}
%\begin{theorem} if $\A$ is a $BLO$ that is strongly regular, then the lattice of 
%\end{theorem}
%\end{definition}
If $L$ is not relatively complemented, then (2) and (3) above are not equivalent; but if it is relatively complemented then they are equivalent. 
A lattice with the property that every interval is complemented is called a relatively complemented lattice. 
In other words, a relatively complemented lattice is characterized by the property that for every element $a$ in an interval $[c,d]=\{x: c\leq x\leq d\}$
there is an element $b$, such that $a\lor b=d$ and $a\land b=c$. 
Such an element is called a complement; it may not be unique, but if the lattice is bounded then relative complements in $[a, 1]$ are just 
complements, and in case of distributivity such complements are 
unique.
In arbitrary lattices the lattice of ideas may not be isomorphic to the lattice of congruences, the 
following theorem gives a sufficient and necessary condition for this to hold.
The theorem is  a classic due to Gratzer and Schmidt.

\begin{theorem} For the correspondence between congruences and ideals to be an 
isomorphism it is necessary and sufficient that $L$ is distributive, relatively complemented with a minimum $0$.
\end{theorem}
\begin{proof} {\bf Sketch} Clearly the ideal corresponding to the identity relation is the $0$ ideal. 
Since every ideal of $L$ is a congruence class under some homomorphism, we obtain distributivity.
To show relative complementedness,  it suffices to show that if $b<a$, then $b$ has a complement in the interval $[0,a]$. 
Let $I_{a,b}$ be the ideal which consists of all $u$ with 
$u\equiv 0(Theta_{a,b})$. $V_{a,b}$ is 
a congruence class under precisely one relation, hence $a\equiv b mod(\Theta[V_{a,b}])$.  Hence for some $v\in I_{a,b}$ we have 
$b\lor v=a$ and $b\land v=0$. 
Conversely, we have every ideal is a congruence class under at most one congruence relation, and of course under at least one.
\end{proof}

In case or relative complementation, we have
\begin{theorem} the following are equivalent
\begin{enumarab}
\item $L$ is strongly regular
\item Every principal ideal of $L$ is generated by a an elemnt in $\Zd L$
\item $\delta(\A)$ is semisimple
\end{enumarab}
\end{theorem}
\begin{proof} Easy
\end{proof}

We push the duality a step futher esatablishing a correspondence between open (closed) sets of $BLO$s and open subsets of its dual.
An ideal $I$ in $\A$ is regular if $\Ig^{\A}(I\cap \Zd\A)=I$.
\begin{theorem} There is an isomomorphism between the set of all regular ideals in $\Gamma(X, \delta)$ 
onto the lattice of open subsets of $X.$
\end{theorem}
\begin{proof} For $\sigma\in \Gamma(X,\delta),$ let $[\sigma]=\{x\in X: \sigma(x)\neq 0_x\}$. For $U\subseteq X$, let 
$J[U]=\{\sigma\in \Gamma(X, \delta): [\sigma]\subseteq U\}.$ 
Then $J\mapsto U[J]$ is an isomorphism, its inverse is $U[J]=\bigcup\{[\sigma]: \sigma\in J\}.$
\end{proof}

%It is easy to see that locally finite algebras are nice. 
%For a class of algebras $K$ we say that $K$ has $ES$ if epimorphisms (in the categorial sense) are surjective.
Note that a simple lattice is necessarily strongly regular (and hence regular), but the converse is not true, even in the case of strong regularity.
There are easy examples.
As an application to our duality theorem established above, we show that certain properties can extend from simple structures to 
strongly regular ones. The natural question that bears an answer is how far are strongly regular algebras from simple algebras; 
and the answer is: pretty 
far. 
For example in cylindric algebras any non-complete theory $T$ in a first order language gives 
rise to a strongly regular $\omega$-dimensional algebra, namely, $\Fm_T$, that is not simple.
%We will show that $ES$ fails in the class of simple infinite dimensional cylindric algebras. 
%To prove our result we use Stephen Comer's sheaf theoretical construction \cite{Co1}, \cite{Co2} and Judit's 
%Madar\'asz construction \cite{M} for the semisimple case.

%\section{Applications}

$ES$ abreviates that epimorphisms (in the categorial sense) are surjective. Such abstract property 
is equivalent to the well-known Beth definability property 
for many abstract logics, including fragments of first order logic, and multi-modal logics. 

In fact, it applies to any algebraisable logic (corresponding to a quasi-variety) regarded
as a concrete category.  This connection was established by N\'emeti. 
As an application, to our hitherto established duality, we have: 

\begin{theorem} Let $V$ be a class of distributive bounded lattices such that the simple lattices in $V$ 
have the amalgamation property $(AP)$.
Assume that there exist strongly regular lattices $\A,\B\in V$ and an epimorphism $f:\A\to \B$ that is not onto.
Then $ES$ fails in the class of simple lattices
\end{theorem}

\begin{demo}{Proof} Suppose, to the contrary that $ES$ holds for simple algebras.
Let $f^*:\A\to \B$ be the given epimorphism that is not onto. We assume that $\A^d=(X,\L)$ and $\B^d=(Y,\G)$ 
are the corresponding dual sheaves over the Priestly  spaces $X$ and $Y$ and by  duality that 
$(h,k)=H:(Y,\G)\to (X,\L)$ is a monomorphism. Recall that $X$ is the set of prime ideals in $Zd\A$, and similarly for $Y$.
We shall first prove
\begin{enumroman}
\item $h$ is one to one
\item for each $y$ a maximal ideal in $\Zd\B$, $k(y,-)$ is a surjection of the stalk over $h(y)$ onto the stalk over $y$.
\end{enumroman}
%This part of the proof is identical to \cite{Co1} Theorem 5.3; it is not a generalization. 
Suppose that $h(x)=h(y)$ for some $x,y\in Y$. Then $\G_x$, $\G_y$ and $\L_{hx}$ are simple algebra, 
so there exists a simple $\D\in V$ and monomorphism $f_x:\G_x\to \D$ and $f_y:\G_y\to \D$ such that
$$f_x\circ k_x=f_y\circ k_y.$$
Here we are using that the algebras considered are strongly regular, and that the simple algebras have $AP$.
Consider the sheaf $(1,D)$ over the one point space $\{0\}=1$ and sheaf morphisms 
$H_x:(\lambda_x,\mu):(1,D)\to (Y,\G)$ and $H_y=(\lambda_y, v):(1,D)\to (Y,\G)$
where $\lambda_x(0)=x$ $\lambda_y(0)=y $ $\mu_0=f_x$ and $v_0=f_y$. The sheaf $(1,\D)$ is the space dual to $\D\in V$
and we have $H\circ H_x=H\circ H_y$. Since $H$ is a monomorphism $H_x=H_y$ that is $x=y$.
We have shown that $h$ is one to one.
Fix $x\in Y$. Since, we are assuming that  $ES$ holds for simple algebras of $V,$ in order to show that 
$k_x:\L_{hx}\to \G_x$ is onto, it suffices to show that $k_x$ is an epimorphism.  
Hence suppose that $f_0:\G_x\to \D$ and $f_1:\G_x\to \D$ for some simple $\D$ such that $f_0\circ k_x=f_1\circ k_x$.
Introduce sheaf morphisms 
$H_0:(\lambda,\mu):(1,\D)\to (Y,\G)$ and $H_1=(\lambda,v):(1,\D)\to (Y,\G)$
where $\lambda(0)=x$, $\mu_0=f_0$ and $v_0=f_1$. Then $H\circ H_0=H\circ H_1$, 
but $H$ is a monomorphism, so we have $H_0=H_1$ from which 
we infer that $f_0=f_1$. 

We now show that (i) and (ii) implies that $f^*$ is onto, which is a contradiction.
%In this part we follow Comer \cite{Co1} Lemma 5.1.
Let $\A^d=(X,\L)$ and $\B^d=(Y,\G)$. It suffices to show that $\Gamma((f^*)^d)$ is onto (Here we are taking a double dual) . 
So suppose $\sigma\in \Gamma(Y,\G)$. For each $x\in Y$, 
$k(x,-)$ is onto so $k(x,t)=\sigma(x)$ for some $t\in \L_{h(x)}$. That is $t=\tau_x(h(x))$ for some 
$\tau_x\in \Gamma(X,\G)$. Hence there is a clopen neighborhood $N_x$ of $x$ such that 
$\Gamma(f^*)^d)(\tau_x)(y)=\sigma(y)$ for all $y\in N_x$. 
Since $h$ is one to one and $X,Y$ are Boolean spaces, we get that $h(N_x)$ is clopen in $h(Y)$ and there is a 
clopen set $M_x$ in $X$ such that $h(N_x)=M_x\cap h(Y)$. Using compactness, there exists a partition of
$X$ into clopen subsets $M_0\ldots M_{k-1}$ and sections $\tau_i\in \Gamma(M_i,L)$ such that
$$k(y,\tau_i(h(y))=\sigma(y)$$ 
wherever $h(x)\in M_i$ for $i<k$. Defining 
$\tau$ by $\tau(z)=\tau_i(z)$ whenever $z\in M_i$ $i<k$, it follows that $\tau\in \Gamma(X,\L)$ and $\Gamma((f^*)^d)\tau=\sigma$.
Thus $\Gamma((f^*)^d)$ is onto $\Gamma(\B^d)$, and we are done. 
\end{demo}
And as an application, using known results, we readily obtain:

\begin{corollary}
\begin{enumarab}
\item Epimorphisms are not surjective in simple cylindric algebras, quasipolyadic algebras and Pinters algebras of infinite dimension
\item Epimorphisms are not surjecive in simple cylindric lattices of infinite dimension
\end{enumarab}
\end{corollary}

\begin{proof} (1) Cf. \cite{AUU} where two strongly regular algebras $\A\subseteq \B$ 
are constructed such that the inclusion is an epimorphism that is not 
surjective. 

(2) In a preprint of ours two strongly regular algebras $\A\subseteq \B$ 
are constructed, and the inclusion is not an epimorphism
\end{proof}
There is a very thin line between superamalgamation $(SUPAP)$ and strong amalgamation $(SAP)$. 
However, Maksimova and Shelah constructed varieties of $BAO$s with $SAP$ but not $SUPAP$, the latter is a variety of representable cylindric 
algebras. The second item of the next corollary makes one cross this line.
\begin{corollary}
\begin{enumarab}
\item  Let $V$ be a variety of $BAO$s such that every semisimple algebra is regular. Then if $ES$ holds for simple algebras, the it holds for 
semisimple algebras.
\item Let $V$ be a variety that has the strong amalgamation property, such that the simple algebras have $ES$. 
Then $V$ has the superamalgamation property.
\end{enumarab}
\end{corollary}
\begin{proof} We only prove the second part. If $SUPAP$ fails in $V$, then $ES$ does, because $V$ has $SAP$ 
and both together are equivalent to $SUPAP$, but then $ES$ fails in simple algebras 
and this is a contradiction.
\end{proof}

It is known that $ES$ fails for semisimple cylindric algebras of infinite dimension.
In view of the first part of the previous corollary, the next example gives a necessary condition for this. 
But first a definition.  An epimorphism $f:\A\to \B$ is said to be conformal if 
$f(\Zd\A)\subseteq \Zd\B$.

\begin{example}
Let $\C$ be a subdirectly indecomposable cylindric algebra of dimension $\alpha$. Let $I$ be the set of all finite subsets 
of subsets of $\alpha$. Let $F$ be an ultraflier on $I$ such that 
$X_{\Gamma}=\{\Delta\in I:\Gamma\subseteq \Delta\}\in F$ for all $\Gamma\in I$. Then the epimorphism ${}^I\C/ \to  {}^I \C/F$ induced by $F$ 
is not conformal.
\end{example}
The above example actually shows that semisimple algebras need not be regular, and moroever the stalks of the dual space of a semisimple algebra may not 
be even subdirectly indecomposable.

%\end{document}

\end{document}